\documentclass[10pt]{amsart}

\usepackage[T1]{fontenc}
\usepackage{lmodern}
\usepackage{amssymb,microtype}
\usepackage[
  colorlinks=true,
  linkcolor=blue,
  citecolor=blue,
  urlcolor=blue
]{hyperref}

\newtheorem{theorem}{Theorem}
\newtheorem{lemma}[theorem]{Lemma}

\theoremstyle{remark}
\newtheorem{remark}[theorem]{Remark}

\newcommand{\Z}{\mathbb{Z}}
\newcommand{\Q}{\mathbb{Q}}
\newcommand{\C}{\mathbb{C}}
\newcommand{\R}{\mathbb{R}}
\newcommand{\defeq}{\stackrel{\mathrm{def}}{=}}
\newcommand{\Visible}{\mathrm{Visible}}
\newcommand{\Invisible}{\mathrm{Invisible}}

\hypersetup{
  pdftitle={Density one for lattice point visibility along polynomials with at least two distinct roots},
  pdfauthor={Abraham Lobsenz},
  pdfsubject={Lattice point visibility along polynomials},
  pdfkeywords={lattice point visibility, polynomial visibility, visible lattice points, density}
}

\title[Polynomial visibility has density one]{Density one for lattice point visibility along polynomials with at least two distinct roots}
\author{Abraham Lobsenz}
\address{Department of Mathematics, Dartmouth College, 6188 Kemeny Hall,
27 N. Main Street, Hanover, New Hampshire 03755-3551, USA}
\email{abe.e.lobsenz.29@dartmouth.edu}
\date{September 7, 2026}

\begin{document}

\begin{abstract}
We prove that every nonzero integer polynomial with at least two distinct
complex roots has lattice point visibility density one, resolving the
generalized form of the Visibility Density Conjecture for nonzero polynomials.
This extends the origin-passing case established by Chaubey, Pandey, and Regavim. Visibility is taken along the curves
$y=tF(x)$ with rational $t$, with a point visible if no positive lattice
point on the same curve has a smaller horizontal coordinate. The proof uses a greatest common divisor cutoff to reduce the problem to
finitely many equations of the form $F(b)=qF(a)$, with $0<q<1$;
for each equation, the positive integers $a$ admitting a positive integer
solution $b<a$ form a set of density zero. This
elementary argument removes the proper-power hypothesis of earlier work
without requiring estimates uniform in the ratio.
\end{abstract}

\maketitle

\section{Introduction}\label{sec:intro}
A classical question in number theory asks what proportion of integer lattice
points are visible from the origin along straight lines. In the positive
quadrant, visibility is equivalent to coprimality of the coordinates, giving
the density $6/\pi^2$. Following Chaubey--Pandey \cite{CP} and
Lobsenz--Phillips \cite{LP}, we ask the same question for polynomial lines
of sight $y=tF(x)$, with $t\in\Q$. We prove that every nonzero integral
polynomial with at least two distinct complex roots has visibility density one.

Theorem~\ref{thm:main} is accompanied by a complete Lean formalization
\cite{Lean}, with no admitted proofs or additional mathematical axioms.

\smallskip
We use the visibility definition and counting notation of
\cite[Definition~2.1 and \S2]{LP}.
For $F\in\Z[x]$, a point $(a,h)\in\Z_{>0}^2$ is \emph{visible along $F$}
if there is $t\in\Q$ with $h=tF(a)$ and $a$ is the least positive
integer $u$ for which $tF(u)$ is a positive integer. Otherwise it is
\emph{invisible along $F$}. A point $(b,k)\in\Z_{>0}^2$ \emph{blocks}
$(a,h)$ if $b<a$ and $h=tF(a)$, $k=tF(b)$ for the same $t\in\Q$.
For a positive integer $N$, write $[1,N]\defeq\{n\in\Z:1\leq n\leq N\}$, and put
\begin{equation}\label{eq:visible-invisible}
\begin{aligned}
 \Visible_F(N)&\defeq\{(a,h)\in[1,N]^2:(a,h)\text{ is visible along }F\},\\
 \Invisible_F(N)&\defeq\{(a,h)\in[1,N]^2:(a,h)\text{ is invisible along }F\}.
\end{aligned}
\end{equation}
These sets partition $[1,N]^2$. The \emph{visibility density} is
\begin{equation}\label{eq:density}
 D(F)\defeq\lim_{N\to\infty}\frac{\#\Visible_F(N)}{N^2},
\end{equation}
whenever the limit exists.

\begin{theorem}\label{thm:main}
Let $F\in\Z[x]\setminus\{0\}$ have at least two distinct complex roots.
Then $D(F)=1$; equivalently,
\[
                         \#\Invisible_F(N)=o(N^2).
\]
\end{theorem}

This proves the generalized form of the Visibility Density Conjecture stated in
\cite[Conjecture~1.1]{LP} for nonzero $F$, extending the origin-passing conjecture of
Chaubey--Pandey \cite[Conjecture~1.6]{CP}, resolved by Chaubey, Pandey,
and Regavim \cite{CPR}. Lobsenz--Phillips proved the assertion for
$F=f^m$ with $m\geq2$ and $f$ having at least two distinct roots
\cite[Theorems~1.2--1.3]{LP}, with quantitative bounds sharpened for
quadratic $f$ in \cite{LPQ}.
The nonzero condition is necessary: $F=0$ vanishes at every complex number
but has no visible positive lattice points, so the literal formulation of
\cite[Conjecture~1.1]{LP} needs this qualification.

\begin{remark}[Sharpness]
Nonzero constants have visibility density zero, since only points with $a=1$
are visible. Every nonconstant integer polynomial with just one distinct
root has the form $c(ux+v)^d$, where $c\in\Z\setminus\{0\}$,
$u,d\in\Z_{>0}$, $v\in\Z$, and $\gcd(u,v)=1$; in this case
\cite[\S1, equation~(7)]{LP} gives
\[
 D\bigl(c(ux+v)^d\bigr)=
 \prod_{\substack{p\text{ prime}\\p\nmid u}}\left(1-\frac{1}{p^{d+1}}\right)<1.
\]
Thus the two-root hypothesis is sharp, including in degree one.
\end{remark}

Earlier joint work used Pila's bound \cite{LP,Pila} and Pell equations
\cite[\S\S3--4]{LPQ} to count blocking pairs. Here a GCD cutoff reduces
the density question to finitely many fixed ratios, each handled by an
elementary argument. This removes the proper-power hypothesis; the earlier
quantitative estimates remain stronger in the cases they cover.

Ahuja \cite[Definition~3.1 and Lemma~3.2]{Ahuja} also uses the GCD of a
polynomial value and a height, giving a sufficient visibility criterion
when this GCD is one. Our cutoff allows all bounded GCD values and
controls the remaining points at arbitrarily small density cost.

\section{Arithmetic reduction}\label{sec:reduction}

Replacing $F$ by $-F$ leaves visibility unchanged, by replacing $t$ by
$-t$, as observed in \cite[\S1]{LP}. We may therefore assume that the
leading coefficient is positive. Following \cite[\S2]{LP}, choose
$n_F\geq1$ such that, whenever $a>n_F$,
\begin{equation}\label{eq:record}
             F(a)>\max\bigl(0,F(1),\ldots,F(a-1)\bigr).
\end{equation}
Such a choice is possible because $F$ is nonconstant (it has at least two
distinct roots), and a nonconstant polynomial with positive leading
coefficient is eventually strictly increasing and tends to infinity. The columns $a\leq n_F$ contain only $O_F(N)$ points.

Put $\gamma_F(a,h)\defeq\gcd(|F(a)|,h)$. Let $(a,h)$ be invisible
with $a>n_F$, and write $\gamma=\gamma_F(a,h)$. Its unique possible scale is
$t=h/F(a)>0$, and a blocking point $(b,k)$ satisfies
\[
 1\leq b<a,\qquad k=\frac{hF(b)}{F(a)}\in\Z_{>0},\qquad
 0<F(b)<F(a).
\]
As in the proof of \cite[Lemma~2.3]{LP}, the identity $hF(b)=kF(a)$ gives a
divisibility condition: since $F(a)/\gamma$ and $h/\gamma$ are coprime,
$F(a)/\gamma\mid F(b)$.
Thus $F(b)=jF(a)/\gamma$ for some integer $1\leq j<\gamma$.
Writing $j/\gamma=r/s$ in lowest terms gives
\begin{equation}\label{eq:finite-ratio}
       F(b)=\frac{r}{s}F(a),\qquad
       1\leq r<s,\quad\gcd(r,s)=1,\quad s\mid\gamma.
\end{equation}
In particular, bounding $\gamma$ leaves only finitely many ratios.

\section{The GCD cutoff}\label{sec:cutoff}
\begin{lemma}\label{lem:tight}
For every nonzero $F\in\Z[x]$ and every $\varepsilon>0$ there is an
integer $M\geq1$ such that, for every $N\geq1$,
\[
 \frac1{N^2}
 \#\{(a,h)\in[1,N]^2:\gamma_F(a,h)>M\}<\varepsilon.
\]
\end{lemma}
\begin{proof}
Write $d=\deg F$. Choose an integer $K\geq2$ large enough that no prime
$p>K$ divides every coefficient of $F$ and $2d/K<\varepsilon/2$.
Such a choice is possible since $F$ has a nonzero coefficient, so only
finitely many primes divide every coefficient of $F$. For any prime $p>K$, the congruence $F(a)\equiv0\pmod p$
has at most $d$ residue classes. Hence
\[
 \#\{(a,h)\in[1,N]^2:p\mid F(a),\ p\mid h\}
 \leq d\left(\frac Np+1\right)\frac Np.
\]
Any prime dividing $h$ is at most $N$, and for $p\leq N$ we have
$N/p+1\leq2N/p$. Summing over $K<p\leq N$ and enlarging the sum to
all integers gives the bound, uniform in $N$,
\begin{equation}\label{eq:large-primes}
 \frac{\#\{(a,h)\in[1,N]^2:\exists p>K,\ p\mid\gamma_F(a,h)\}}{N^2}
 \leq 2d\sum_{n>K}\frac1{n^2}\leq\frac{2d}{K}.
\end{equation}
For any fixed positive integer $E$,
the proportion of heights divisible by $p^E$ is at most $p^{-E}$.
Outside a set of pairs of proportion at most
$\sum_{p\leq K}p^{-E}$, none of these small-prime powers divides $h$.
If neither exceptional event occurs, then
\[
       \gamma_F(a,h)\leq M\defeq\prod_{p\leq K}p^{E-1}.
\]
Finally, choose $E$ with $\sum_{p\leq K}p^{-E}<\varepsilon/2$. Combining the two exceptional
sets proves the claim.
\end{proof}

\section{Fixed ratios are sparse}\label{sec:ratios}
\begin{lemma}\label{lem:sparse}
Let $F\in\R[x]$ have positive leading coefficient and at least two
distinct complex roots. For each $q\in(0,1)$, the set
\[
 \mathcal A_q(F)\defeq
 \{a\in\Z_{>0}:\exists b\in\Z_{>0},\ b<a,\ F(b)=qF(a)\}
\]
has natural density zero.
\end{lemma}
\begin{proof}
Put $d=\deg F\geq2$. Write
\[
 F(x)=c_d x^d+c_{d-1}x^{d-1}+O_F(x^{d-2}),\quad
 \alpha=q^{1/d}\in(0,1),\quad
 \beta=\frac{c_{d-1}(\alpha-1)}{dc_d}.
\]
For any fixed $\eta>0$, the polynomials $F(\alpha x+\beta\pm\eta)-qF(x)$ have degree $d-1$ and
leading coefficients $\pm dc_d\alpha^{d-1}\eta$. Consequently, for large $a$,
\[
 F(\alpha a+\beta-\eta)<qF(a)<F(\alpha a+\beta+\eta).
\]
For any fixed integer $B\geq1$, we have
$qF(a)>\max_{1\leq j\leq B}F(j)$ for all sufficiently large $a$,
so every solution then has $b>B$.
Also $\alpha a+\beta-\eta\to\infty$, so $b$ and both comparison
arguments eventually lie in the range where $F$ is strictly increasing.
The displayed inequalities therefore give
$\alpha a+\beta-\eta<b<\alpha a+\beta+\eta$.
As $\eta$ is arbitrary, this proves the following uniform statement:
\begin{equation}\label{eq:asymptote}
\begin{aligned}
 &\forall\eta>0\ \exists a_0\geq1\ \forall a,b\in\Z_{>0}:\\
 &\qquad a>a_0,\quad b<a,\quad F(b)=qF(a)
       \quad\Longrightarrow\quad |b-\alpha a-\beta|<\eta.
\end{aligned}
\end{equation}

Suppose first that $\alpha$ is irrational. Fix an integer $L\geq1$ and set
\[
 \delta_L=\min_{1\leq j\leq L}\operatorname{dist}(j\alpha,\Z)>0.
\]
By \eqref{eq:asymptote}, beyond a fixed threshold every solution pair
satisfies $|b-\alpha a-\beta|<\delta_L/3$. Two such solutions with
$1\leq a'-a\leq L$ would give
\[
 \operatorname{dist}\bigl((a'-a)\alpha,\Z\bigr)
 \leq |(b'-b)-\alpha(a'-a)|<2\delta_L/3,
\]
a contradiction. Thus every block of $L+1$ consecutive sufficiently
large integers contains at most one member of $\mathcal A_q(F)$.
Its upper density is at most $1/(L+1)$; letting $L\to\infty$ proves
the assertion.

Suppose next that $\alpha$ is rational. Choose an integer $D\geq1$
such that $D\alpha\in\Z$. Then $b-\alpha a\in D^{-1}\Z$.
Equation \eqref{eq:asymptote} shows that, if arbitrarily large solutions
exist, necessarily $\beta\in D^{-1}\Z$ and all sufficiently large
solutions satisfy $b=\alpha a+\beta$ exactly. If infinitely many
solutions existed, the polynomial \(F(\alpha x+\beta)-qF(x)\) would have infinitely many roots. Since a nonzero polynomial has only finitely many roots, this would imply
\[
                   F(\alpha x+\beta)=qF(x)\qquad\text{for every }x\in\C.
\]

We now use the root-permutation argument from the proof of
\cite[Proposition~3.3]{LP}, in its contraction form.
Let $R_F\defeq\{\rho\in\C:F(\rho)=0\}$ and $T(z)=\alpha z+\beta$.
The identity implies $T(R_F)\subseteq R_F$.
As $T$ is injective and $R_F$ finite, $T(R_F)=R_F$. But
\[
          \operatorname{diam}(R_F)=\operatorname{diam}(T(R_F))
                              =\alpha\operatorname{diam}(R_F).
\]
Since $0<\alpha<1$, the diameter is zero, contrary to the existence of
two distinct roots. In this case $\mathcal A_q(F)$ is in fact finite.
\end{proof}

\begin{remark}
For polynomials of degree at least two, the two-distinct-roots hypothesis
is used only in the rational-$\alpha$ case of Lemma~\ref{lem:sparse}.
\end{remark}

\section{Completion of the proof}
\begin{proof}[Proof of Theorem~\ref{thm:main}]
Fix $\varepsilon>0$ and choose $M$ from Lemma~\ref{lem:tight}.
Let
\[
       \mathcal A_M(F)\defeq
       \bigcup_{2\leq s\leq M}\ \bigcup_{\substack{1\leq r<s\\\gcd(r,s)=1}}
                       \mathcal A_{r/s}(F).
\]
This is a finite union of sets of density zero by Lemma~\ref{lem:sparse},
so $\#(\mathcal A_M(F)\cap[1,N])=o_{F,M}(N)$.
By \eqref{eq:finite-ratio}, every invisible point in $[1,N]^2$ either
has $a\leq n_F$, has $\gamma_F(a,h)>M$, or has $a\in\mathcal A_M(F)$.
Consequently,
\[
 \#\Invisible_F(N)
 \leq n_FN+\varepsilon N^2+N\,\#(\mathcal A_M(F)\cap[1,N]).
\]
It follows that
\[
 \limsup_{N\to\infty}\frac{\#\Invisible_F(N)}{N^2}
 \leq 0+\varepsilon+
       \lim_{N\to\infty}\frac{\#(\mathcal A_M(F)\cap[1,N])}{N}
 =\varepsilon.
\]
Since $\varepsilon$ is arbitrary and the left side is nonnegative,
$\#\Invisible_F(N)=o(N^2)$, and hence $D(F)=1$.
\end{proof}

\section*{Formalization}
The formalization \cite{Lean} uses Lean~4.28.0 and mathlib~v4.28.0
(\href{https://github.com/leanprover-community/mathlib4/tree/8f9d9cff6bd728b17a24e163c9402775d9e6a365}{commit \texttt{8f9d9cf}}),
with a slightly coarser GCD cutoff bound yielding the same density theorem.
The verified version is
\href{https://github.com/abelobsenz/polynomial-visibility-lean/tree/6f83b9e899f2527898601fc3e028e2c016459ae5}{commit \texttt{6f83b9e}}
of the repository:\par
\noindent\url{https://github.com/abelobsenz/polynomial-visibility-lean}.

\section*{Acknowledgments}

The author thanks Tristan Phillips for his mentorship in the development of
the joint papers \cite{LP,LPQ}, whose ideas inform the present work. 

Additionally, the author acknowledges the help of AI (GPT-6-Astra) in proof discovery, Lean formalization, and manuscript drafting. The author assumes all responsibility for the arguments laid out here.

\end{document}